\documentclass[12pt]{amsart}
\usepackage{amssymb}
\usepackage{verbatim}
\usepackage{graphicx}
\usepackage[cmtip,all]{xy}
\usepackage{amsmath}
\usepackage{amsfonts,amssymb}
\usepackage{times}
\usepackage{amscd}
\usepackage{epsfig}
\usepackage{float}
\usepackage{cleveref}
\usepackage{amsthm}
\usepackage{latexsym}
\usepackage{imakeidx}
\usepackage[active]{srcltx}

\newtheorem{theo}{Theorem}[section]
\newtheorem{thm}[theo]{Theorem}
\newtheorem{lem}[theo]{Lemma}
\newtheorem{cor}[theo]{Corollary}

\newtheorem{prop}[theo]{Proposition}

\newtheorem*{thmM}{Main Theorem}

\theoremstyle{definition}
\newtheorem{dfn}[theo]{Definition}

\theoremstyle{remark}

\numberwithin{equation}{section}

\def\R{\mathbb{R}}
\def\Z{\mathbb{Z}}
\def\Mc{\mathcal{M}}
\newcommand{\ql}{quadrilateral{}}

\newcommand{\md}{\mathcal{MD}}
\newcommand{\fup}{\mathcal{P}}
\newcommand{\qcp}{\mathrm{QCP}}

\newcommand{\coc}{\mathrm{co}}

\newcommand{\cmld}{\mathrm{CML(\mathcal D)}}

\newcommand{\ta}{\mathrm{T\!ag}}

\newcommand{\cpd}{\mathcal{CPD}}

\newcommand{\zc}{\mathcal{Z}}

\newcommand{\C}{\mathbb{C}}

\newcommand{\disk}{\mathbb{D}}
\newcommand{\cdisk}{\ol{\mathbb{D}}}

\newcommand{\ol}{\overline}

\newcommand{\sm}{\setminus}

\newcommand{\Tc}{\mathcal{T}}

\newcommand{\n}{\ol{n}}
\newcommand{\hell}{\hat{\ell}}

\newcommand{\qml}{\mathrm{QML}}
\newcommand{\bd}{\mathrm{Bd}}

\newcommand{\lam}{\mathcal{L}}

\newcommand{\hlam}{\mathcal{\widehat L}}

\newcommand{\lamm}{\mathcal{L}^m}

\newcommand{\fqcp}{\mathcal{QCP}}

\newcommand{\ch}{\mathrm{CH}}

\newcommand{\si}{\sigma}
\newcommand{\ph}{\varphi}
\newcommand{\uc}{\mathbb{S}}

\newcommand{\mc}{\mathfrak{C}}

\newcommand{\g}{\mathfrak{g}}

\newcommand{\mD}{\mathcal{D}}

\newcommand{\mf}{\mathfrak{m}}

\def\L{\mathbb{L}}

\renewcommand\le{\leqslant}

\def\0{\varnothing}

\begin{document}

\date{February 7, 2017}

\title[Models for spaces of dendritic polynomials]{Models for spaces of dendritic polynomials}

\author[A.~Blokh]{Alexander~Blokh}

\thanks{The first and the third named authors were partially
supported by NSF grant DMS--1201450}

\author[L.~Oversteegen]{Lex Oversteegen}

\author[R.~Ptacek]{Ross~Ptacek}

\author[V.~Timorin]{Vladlen~Timorin}


\thanks{The fourth named author has been supported by the Russian Academic
Excellence Project '5-100'.}

\address[Alexander~Blokh and Lex~Oversteegen]
{Department of Mathematics\\ University of Alabama at Birmingham\\
Birmingham, AL 35294}

\address[Ross~Ptacek and Vladlen~Timorin]
{Faculty of Mathematics\\
National Research University Higher School of Economics\\
6 Usacheva str., Moscow, Russia, 119048}

\email[Alexander~Blokh]{ablokh@math.uab.edu}
\email[Lex~Oversteegen]{overstee@math.uab.edu}
\email[Ross~Ptacek]{rptacek@uab.edu}
\email[Vladlen~Timorin]{vtimorin@hse.ru}



\begin{abstract}
Complex 1-variable polynomials with connected Julia sets and only repelling periodic points are called \emph{dendritic}.
By results of Kiwi, any dendritic polynomial is semi-conjugate to a topological polynomial whose topological Julia set is a dendrite.
We construct a continuous map of the space of all cubic dendritic polynomials onto a laminational model that is a quotient space of a subset of the closed bidisk.
This construction generalizes the ``pinched disk'' model of the Mandelbrot set due to Douady and Thurston.
It can be viewed as a step towards constructing a model of the cubic connectedness locus.
\end{abstract}

\maketitle

\section{Introduction}\label{s:intro}

The \emph{parameter space} of complex degree $d$ polynomials is by definition the space of affine conjugacy classes of these polynomials.
Equivalently, one can talk about the space of all \emph{monic central polynomials} of degree $d$, i.e. polynomials of the form
$z^d+a_{d-2}z^{d-2}+\dots+a_0$.
Any polynomial is affinely conjugate to a monic central polynomial.
An important set is the \emph{connectedness locus} $\Mc_d$ consisting of classes of all degree $d$ polynomials $P$, whose Julia sets $J(P)$ (equivalently, whose \emph{filled Julia sets} $K(P)$) are connected.
General properties of the connectedness locus $\Mc_d$ have been studied for quite some time.
For instance, it is known that $\Mc_d$ is a compact cellular set in the parameter space of complex degree $d$ polynomials
(this was proven in \cite{BrHu} in the cubic case and in \cite{la89} for higher degrees, see also \cite{bra86};
by definition, following M. Brown \cite{bro60,bro61}, a subset of a Euclidean space $\R^n$ is \emph{cellular}\index{cellular set} if its complement in the sphere $\R^n\cup \{\infty\}$ is an open topological cell).

For $d=2$, a monic centered polynomial takes the form $P_c(z)=z^2+c$,
and the parameter space of quadratic polynomials can be identified with the plane of complex parameters $c$.
Clearly, $P_c(z)$ has a unique critical point $0$ and a unique critical value $c$ in $\C$.
Thus, we can say that polynomials $P_c(z)$ are parameterized by their critical values.
The quadratic connectedness locus is the famous \emph{Mandelbrot set}\index{Mandelbrot set} $\Mc_2$, identified with the set of complex numbers $c$ not escaping to infinity under iterations of the polynomial $P_c(z)$.
The Mandelbrot set $\Mc_2$ has a complicated self-similar structure.

\subsection{A combinatorial model for $\Mc_2$}
The  ``pinched disk'' model for $\Mc_2$ is due to Douady and Thurston \cite{dou93, thu85}.
To describe their approach to the problem of modeling $\Mc_2$, we first describe \emph{laminational} models of polynomial Julia sets (we follow \cite{bl02}) and then use them to interpret results of \cite{dou93, thu85}.
We assume basic knowledge of complex dynamics (a detailed description is given later).

Let $\uc$ be the unit circle in $\C$, consisting of all complex numbers of modulus one.
We write $\si_d:\uc\to\uc$ for the restriction of the map $z\mapsto z^d$.
We identify $\uc$ with $\R/\Z$ by the mapping taking an \emph{angle} $\theta\in\R/\Z$ to the point $e^{2\pi i\theta}\in\uc$.
Under this identification, we have $\si_d(\theta)=d\theta$.
We will write $\disk$ for the open unit disk $\{z\in\C\,|\, |z|<1\}$.

Given a complex polynomial $P$, we let $U_\infty(P)$ denote the set $\C\sm K(P)$.
This set is called the \emph{basin of attraction of infinity} of $P$.
Clearly, $\ol{U_\infty(P)}=U_\infty(P)\cup J(P)$.
If the Julia set $J(P)$ is locally connected, then it is connected, and the Riemann map
$\Psi:\C\sm \cdisk\to U_\infty(P)$ can be continuously extended to a map $\ol{\Psi}:\C\sm \disk\to \ol{U_\infty(P)}$.
This gives rise to a map $\psi=\ol{\Psi}|_{\uc}$, which semiconjugates $\si_d:\uc\to \uc$ with $P|_{J(P)}$.
Define an equivalence relation $\sim_P$ on $\uc$ so that $x \sim_P y$ if and only if $\psi(x)=\psi(y)$.
Then $\uc/{\sim_P}$ and $J(P)$ are homeomorphic, and the homeomorphism in question conjugates the map $f_{\sim_P}$ induced on $\uc/{\sim_P}$ by $\si_d$, and $P|_{J(P)}$.
It is not hard to see that the convex hulls of $\sim_P$-classes are disjoint in $\cdisk$. 

A productive idea is to consider equivalences relations $\sim$ whose properties are similar to those of $\sim_P$.
These properties will be stated precisely below.
Such equivalence relations are called \emph{laminational equivalence relations (of degree $d$}; if $d=2$ they are said
to be \emph{quadratic}, and if $d=3$ they are said to be \emph{cubic}.
The maps $f_\sim:\uc/\sim\to \uc/\sim$ induced by $\si_d$ are called \emph{topological polynomials} of degree $d$; again,
if $d=2$ they are called \emph{quadratic} and if $d=3$ they are called \emph{cubic}.
For brevity, in what follows we will talk about ``$\sim$-classes'' instead of ``classes of equivalence of $\sim$''.

An important geometric interpretation of a laminational equivalence relation $\sim$ is as follows.
For any $\sim$-class $\g$, take its convex hull $\ch(\g)$.
Consider the edges of all such convex hulls; add all points of $\uc$ to this collection of chords.
The obtained collection of (possibly, degenerate) chords in the unit disk is denoted
by $\lam_\sim$ and is called a \emph{geodesic lamination (generated by $\sim$)}.
For brevity in what follows we sometimes write ``lamination'' instead of ``geodesic lamination''.
Clearly, $\lam_\sim$ is a closed family of chords.
Let $\ol{ab}$ denote the straight line segment connecting points $a$, $b\in \uc$.
We will never use this notation for pairs of points not in $\uc$.
Recall also that points in $\uc$ are identified with their angles.
Thus, $\ol{0\frac 12}$ always means the chord of $\uc$ connecting the points with angles $0$ and $\frac 12$
(not a half-radius of the unit disk originating at the center).
For any chord $\ell=\ol{ab}$ in the closed unit disk $\cdisk$ set $\si_d(\ell)=\ol{\si_d(a) \si_d(b)}$.
For any $\sim$-class $\g$ and, more generally, for any closed set $\g\subset\uc$, we set $\si_d(\ch(\g))=\ch(\si_d(\g))$.

Recall the construction of Douady and Thurston.
Suppose that a quadratic polynomial $P_c$ has locally connected Julia set.
We will write $G_c$ for the convex hull of the $\sim_{P_c}$-class corresponding to the critical value $c$.
A fundamental theorem of Thurston is that $G_c\ne G_{c'}$ implies that $G_c$ and $G_{c'}$ are disjoint in $\cdisk$. 
Consider the collection of all $G_c$ and take its closure.
The thus obtained collection of chords and inscribed polygons defines a geodesic lamination
$\qml$ introduced by Thurston in \cite{thu85} and called the \emph{quadratic minor
lamination}; moreover, it induces an equivalence relation $\sim_{\qml}$ on $\uc$ \cite{thu85}.
The corresponding quotient space $\Mc^{comb}_2=\uc/\sim_{\qml}$ is a combinatorial model for the boundary of $\Mc_2$.
It is called the \emph{combinatorial Mandelbrot set}.
Conjecturally, the combinatorial Mandelbrot set is homeomorphic to the boundary of $\Mc_2$.
This conjecture is equivalent to the famous MLC conjecture: the Mandelbrot set is locally connected.

\subsection{Dendritic polynomials}
When defining the combinatorial Mandelbrot set, we used a partial association between parameter values $c$ and laminational equivalence relations $\sim_{P_c}$.
In order to talk about $\sim_{P_c}$, we had to assume that $J(P_c)$ was locally connected.
Recall that a \emph{dendrite} is a locally connected continuum that does not contain Jordan curves.
Recall also, that a map from a continuum to a continuum is called \emph{monotone} if under this map
point-preimages (\emph{fibers}) are connected.

\begin{dfn}\label{d:dendr-poly}
A complex polynomial $P$ is said to be \emph{dendritic} if it has connected Julia set and all cycles repelling.
A topological polynomial is said to be \emph{dendritic} if its Julia set is a dendrite.
In that case the corresponding laminational equivalence relation and
the associated geodesic lamination are also said to be \emph{dendritic}.
\end{dfn}

It is known that there are dendritic polynomials with non-locally connected Julia sets.
Nevertheless,
by \cite{kiwi97}, for \emph{every} dendritic polynomial $P$ of degree $d$, there is a monotone
semiconjugacy $m_P$ between $P:J(P)\to J(P)$ and a certain topological polynomial $f_{\sim_P}$
such that the map $m_P$ is one-to-one on all periodic and pre-periodic points of $P$.
Moreover, by \cite{bco11} the map $m_P$ is unique and can be defined in a purely topological way.
Call a monotone map $\ph_P$ of a
connected polynomial Julia set $J(P)=J$ onto a locally connected
continuum $L$ the \emph{finest monotone map of $J(P)$ onto a locally
connected continuum} if, for any monotone $\psi:J\to J'$ with $J'$
locally connected, there is a monotone map $h$ with $\psi=h\circ
\ph_P$. Then it is proven in \cite{bco11} that for \emph{any polynomial} the finest monotone map
on a connected polynomial Julia set semiconjugates $P|_{J(P)}$ to the
corresponding topological polynomial $f_{\sim_P}$ on its topological
Julia set $J_{\sim_P}$ generated by the laminational equivalence
relation possibly with infinite classes $\sim_P$, and that in the dendritic case
this semiconjugacy coincides with the map $m_P$ constructed by Kiwi in \cite{kiwi97}.
Clearly, this shows that $m_P$ is unique.

Thus, $P$ gives rise to a corresponding laminational equivalence relation $\sim_P$ even if $J(P)$ is not locally connected.
If $P_c(z)=z^2+c$ is a quadratic dendritic polynomial, then $G_c$ is defined, and
is a finite-sided polygon inscribed into $\uc$, or a chord, or a point.
A parameter value $c$ is said to be \emph{quadratic dendritic} if $P_c$ is dendritic.
The fundamental results of Thurston \cite{thu85} imply, in particular, that $G_c$ and $G_{c'}$ are either the same or disjoint, for all pairs $c$, $c'$ of dendritic parameter values.
Moreover, the mapping $c\mapsto G_c$ is \emph{upper semi-continuous}
(if a sequence of dendritic parameters $c_n$ converges to a dendritic parameter $c$, then the limit set of the corresponding convex sets $G_{c_n}$ is a subset of $G_c$).
We call $G_c$ the \emph{tag associated to $c$}.

Now, consider the union of all tags of quadratic dendritic polynomials.
This union is naturally partitioned into individual tags (distinct tags are pairwise disjoint!).
Thus the space of tags can be equipped with the quotient space topology induced from the union of tags.
On the other hand, take the set of quadratic dendritic parameters.
Each such parameter $c$ maps to the polygon $G_c$, i.e. to the tag associated to $c$.
Thus each quadratic dendritic parameter maps to the corresponding point of the space of tags.
This provides for a combinatorial (or laminational) model for the set of quadratic dendritic polynomials (or their parameters).

In this paper, we extend these results to cubic dendritic polynomials.

\subsection{Mixed tags of cubic polynomials}
Recall that monic centered quadratic polynomials are parameterized by their critical values.
A combinatorial analog of this parameterization is the association between topological polynomials and their tags.
Tags of quadratic topological polynomials are post-critical objects of the corresponding laminational equivalences.
Monic centered cubic polynomials can be parameterized by a critical value and a co-critical point.
Recall that the \emph{co-critical} point $\omega^*$ of a cubic polynomial $P$ corresponding to a simple critical point $\omega$ of $P$ is defined as a point different from $\omega$ but having the same image under $P$ as $\omega$.
If $\omega$ is a multiple critical point of $P$, then we set $\omega^*=\omega$.
In any case we have $P(\omega^*)=P(\omega)$.
Let $c$ and $d$ be the two critical points of $P$
(if $P$ has a multiple critical point, then $c=d$).
Set $a=c^*$ and $b=P(d)$.
Assuming that $P$ is monic and central, we can parameterize $P$ by $a$ and $b$:
$$
P(z)=b+\frac{a^2(a-3z)}{4}+z^3.
$$
For $P$ in this form, we have $c=-\frac a2$, $d=\frac a2$.
Similarly to parameterizing cubic polynomials by pairs $(a,b)$, we will use the so-called \emph{mixed tags} to parameterize topological cubic dendritic polynomials.

Consider a cubic dendritic polynomial $P$. By the above there exists a laminational equivalence relation
$\sim_P$ and a monotone semiconjugacy  $m_p:J(P)\to \uc/\sim_P$ between $P_{J_P}$ and
the topological polynomial $f_{\sim_P}$.
Given a point $z\in J(P)$, we associate with it the convex hull $G_{P,z}$ of the $\sim_P$-equivalence class represented by the point $m_{P}(z)\in\uc/\sim_P$.
If $P$ is fixed, we may write $G_z$ instead of $G_{P,z}$.
The set $G_z$ is a convex polygon with finitely many vertices, a chord, or a point;
it should be viewed as a combinatorial object corresponding to $z$.
For any points $z\ne w\in J(P)$, the sets $G_z$ and $G_w$ either coincide or are disjoint.

By definition, a \emph{$($critically$)$ marked} (cf \cite{mil12}) cubic polynomial is a triple $(P,c,d)$, where $P$ is a cubic polynomial with critical points $c$ and $d$.
If $P$ has only one (double) critical point, then $c=d$, otherwise $c\ne d$.
In particular, if $c\ne d$, then the triple $(P,c,d)$ and the triple $(P,d,c)$ are viewed as two distinct critically marked cubic polynomials.
When the order of the critical points is fixed, we will sometimes write $P$ instead of $(P,c,d)$.
Critically
marked polynomials do not have to be dendritic (in fact, the notion
is used by Milnor and Poirier \cite{mil12} for hyperbolic polynomials, i.e., in
the situation diametrically opposite to that of dendritic
polynomials). However in this paper whenever we talk about critically marked polynomials
we mean that they are dendritic.

Let $\mathcal{MD}_3$ be the space of all critically marked cubic dendritic polynomials.
With every marked dendritic polynomial $(P,c,d)$, we associate the corresponding \emph{mixed tag}
$$
\ta(P,c,d)=G_{c^*}\times G_{P(d)}\subset\overline{\disk}\times\overline{\disk}.
$$
Here $c^*$ is the co-critical point corresponding to the critical point $c$.

A similar construction can be implemented for any cubic dendritic laminational equivalence relation $\sim$.
Let $C$ and $D$ denote the convex hulls of its critical classes.
Then either $C=D$ is the unique critical $\sim$-class or $C\ne D$ are disjoint.
The sets $C$ and $D$ are called \emph{critical objects} of $\sim$. 
By a \emph{$($critically$)$ marked cubic laminational equivalence relation} we mean a triple $(\sim,C,D)$;
in that case we always assume that $\sim$ is dendritic.
If $C\ne D$, then we define $C^*=\coc(C)$ as the convex hull of the unique $\sim$-class that is distinct from the class $C\cap\uc$ but has the same $\si_3$-image.
If $C=D$, then we set $C^*=C$.
The set $C^*$ is called the \emph{co-critical set of $C$}.
For a marked laminational equivalence relation $(\sim,C,D)$, define its \emph{mixed tag} as
$$ \ta_l(\sim, C, D)=C^*\times \si_3(D)\subset
\overline{\disk}\times\overline{\disk}
$$
We endow the family of products of compact subsets of $\cdisk$ with the
product topology on $\mc(\cdisk)\times \mc(\cdisk)$. It is easy to see
that the map $\ta_l$ is continuous as a map defined on a subset of
$\mc(\cdisk)\times \mc(\cdisk)$. Evidently, the map $\ta_l$ preserves inclusions.

The subscript $l$ in $\ta_l$ indicates that $\ta_l$ acts on marked
laminational equivalence relations unlike the map $\ta$ that acts on
polynomials. These two maps are closely related though: for any marked
dendritic cubic polynomial $(P,c,d)$ and the corresponding marked
laminational equivalence relation $(\sim_P,G_c, G_d)$, we have
$\ta(P,c,d)=\ta_l(\sim_P,G_{c},G_{d})$.

\subsection{Statement of the main result}\label{ss:mainres}
Consider the collection of the sets $\ta(P)$ over all $P\in\md_3$. By
\cite{kiwi97, kiw05}, for any dendritic laminational equivalence
relation $\sim$, there exists a dendritic complex polynomial $P$ with
$\sim=\sim_P$. Thus, equivalently, we can talk about the collection of
mixed tags of all dendritic laminations $\lam_\sim$. In
Theorem~\ref{t:tagusc} we show that the mixed tags $\ta(P)$ are
pairwise disjoint or equal. Let us denote this collection of sets be
$\cmld$ (for \emph{cubic mixed lamination of dendritic polynomials}). Note, that $\cmld$ can be
viewed as (non-closed) ``lamination'' in $\cdisk\times \cdisk$ whose
elements are products of points, leaves or gaps. One can consider
$\cmld$ as the higher-dimensional analog of Thurston's $\qml$ restricted to
dendritic polynomials.

Theorem~\ref{t:tagusc}, in addition, establishes the fact that the
collection of sets $\cmld$ is upper semi-continuous. Let the
\emph{union} of all sets in $\cmld$ be denoted by $\cmld^+\subset
\cdisk\times \cdisk$. It follows that the quotient space of $\cmld^+$,
obtained by collapsing all elements of $\cmld$ to points, is a separable
metric space which we denote by $\mathcal{MD}^{comb}_3$. Denote by
$\pi:\cmld^+\to \mathcal{MD}^{comb}_3$  the corresponding quotient map.

\begin{thmM}
Mixed tags of critically marked polynomials from $\mathcal{MD}_3$ are disjoint or coincide.
The map $\pi\circ \ta: \md_3\to \mathcal{MD}^{comb}_3$ is continuous.
Hence $\mathcal{MD}^{comb}_3$ is a combinatorial model for $\md_3$.
\end{thmM}

This theorem can be viewed as a partial generalization of Thurston's results \cite{thu85} to cubic polynomials.
It is also a first step towards defining a combinatorial model for $\Mc_3$.

\subsection{Previous work and organization of the paper}
Lavaurs \cite{la89} proved that $\Mc_3$ is not locally connected.
Epstein and Yampolsky \cite{EY99} proved that the bifurcation locus in the space of real cubic polynomials is not locally connected either.
This makes the problem of defining a combinatorial model of $\Mc_3$ very delicate.
There is no hope that a combinatorial model would lead to a precise topological model.
Schleicher \cite{sch04} constructed a geodesic lamination modeling the space of \emph{unicritical} 
polynomials,
that is, 
polynomials with a unique multiple critical point.
We have heard of an unpublished old work of D. Ahmadi and M. Rees, in which cubic geodesic laminations were studied, however we have not seen it.
The present paper is based on the results obtained in \cite{bopt16}.
These results are applicable to invariant laminations of any degree.

The paper is organized as follows.
In Section \ref{s:basicdef}, we discuss basic properties of geodesic laminations and laminational equivalence relations.
In Section \ref{s:long}, we recall the results of \cite{bopt16} adapting them to the cubic case.
Finally, Section \ref{s:cudebi} is dedicated to the proof of the main result.

\section{Laminations and their properties}\label{s:basicdef}

By a \emph{chord} we mean a closed segment connecting two points of the unit circle.
If these two points coincide, then the chord is said to be \emph{degenerate}.

\begin{dfn}[Geodesic laminations]\label{d:geolam}
A \emph{geodesic lamination} is a collection $\lam$ of chords called \emph{leaves} that are pairwise disjoint in $\disk$;
it includes all degenerate chords, and must be such that $\lam^+=\bigcup_{\ell\in\lam}\ell$ is closed.
\emph{Gaps} of $\lam$ are the closures of the components of $\disk\sm\lam^+$.
\end{dfn}

We now introduce the notion of a (sibling) \emph{$\si_d$-invariant geodesic lamination}.
This is a slight modification of an invariant geodesic lamination introduced by Thurston \cite{thu85}.
When $d$ is fixed we will often write ``invariant'' instead of ``$\si_d$-invariant''.

\begin{dfn}[Invariant geodesic laminations \cite{bmov13}]\label{d:sibli}
A geodesic lamination $\lam$ is (sibling) \emph{$\si_d$-invariant}
provided that:
\begin{enumerate}
\item for each $\ell\in\lam$, we have $\si_d(\ell)\in\lam$,
\item \label{2}for each $\ell\in\lam$ there exists $\ell^*\in\lam$ so that $\si_d(\ell^*)=\ell$.
\item \label{3} for each $\ell\in\lam$ such that $\si_d(\ell)$ is a
    non-degenerate leaf, there exist $d$ \textbf{pairwise disjoint}
 leaves $\ell_1$, $\dots$, $\ell_d$ in $\lam$ such that
 $\ell_1=\ell$ and
 $\si_d(\ell_i)=\si_d(\ell)$ for all
  $i=2$, $\dots$, $d$.
\end{enumerate}
\end{dfn}

Observe that, since leaves are chords, and chords are closed segments, pairwise disjoint leaves in part (3) of the above definition cannot intersect even on the unit circle (that is, they cannot have common endpoints).

Call the leaf $\ell^*$ in (\ref{2}) a \emph{pullback} of $\ell$.
A \emph{sibling of $\ell$} is defined as a leaf $\ell'\ne \ell$ with $\si_d(\ell')=\si_d(\ell)$.
Thus, part (3) of the definition implies that any leaf with non-degenerate image has $d-1$ pairwise disjoint siblings.
Moreover, these siblings can be chosen to be disjoint from the leaf.
Definition~\ref{d:sibli} implies Thurston's but is slightly more restrictive \cite{bmov13}.

From now on, by \emph{$(\si_d$-$)$invariant geodesic laminations}, we always mean \emph{sibling $\si_d$-invariant geodesic laminations}.
Moreover, for brevity we often talk about laminations meaning \emph{sibling $\si_d$-invariant geodesic} laminations.

\begin{dfn}[Linked chords]\label{d:linchor}
Two \textbf{distinct} chords of $\disk$ are \emph{linked} if they intersect in $\disk$.
We will also sometimes say that these chords \emph{cross each other}.
Otherwise two chords are said to be \emph{unlinked}.
\end{dfn}

A gap $G$ is said to be \emph{infinite $($finite, uncountable$)$} if $G\cap\uc$ is infinite (finite, uncountable).
Uncountable gaps are also called \emph{Fatou} gaps.
For a closed convex set $H\subset \C$, straight segments in the boundary $\bd(H)$ of $H$ are called \emph{edges} of $H$.

\begin{dfn}[Critical sets]\label{d:cristuff}
A \emph{critical chord $($leaf$)$} $\ol{ab}$ of $\lam$ is a chord
(leaf) of $\lam$ such that $\si_d(a)=\si_d(b)$. A gap is
\emph{all-critical} if all its edges are critical. An all-critical gap
or a critical leaf (of $\lam$) is called an \emph{all-critical set} (of
$\lam$). A gap $G$ of $\lam$ is said to be \emph{critical} if it is an
all-critical gap or there is a critical chord contained in the interior of
$G$ except for its endpoints. A \emph{critical set} of $\lam$ is by
definition a critical leaf or a critical gap. We also define a
\emph{critical object} of $\lam$ as a maximal by inclusion critical
set.
\end{dfn}

Given a compact metric space $X$, the space of all its compact subsets with the Hausdorff metric is denoted by $\mc(X)$.
Note that $\lam^+$ is compact for every geodesic lamination $\lam$.
Hence it can be viewed as a point in the space $\mc(\cdisk)$.
The family of sets $\lam^+$ of all invariant
geodesic laminations $\lam$ is compact in $\mc(\cdisk)$. However the set
$\lam^+$ does not always determine the geodesic lamination $\lam$. Indeed,
any geodesic lamination $\lam$ without gaps has $\lam^+=\cdisk$. On the
other hand, already in the cubic case there are two distinct invariant
geodesic laminations $\lam_v$ and $\lam_h$ without gaps. Here $\lam_v$ contains
all vertical chords and $\lam_h$ contains all horizontal chords. Observe
that the corresponding topological Julia sets are arcs with induced
topological polynomials being non-conjugate ``saw-tooth'' maps.
However, $\lam$ itself is a point of $\mc(\mc(\cdisk))$ which determines all
the leaves of $\lam$. For this reason
we consider the family of all invariant geodesic laminations as a subspace of $\mc(\mc(\cdisk))$ with Hausdorff metric;
the notion of convergence of invariant geodesic laminations is understood accordingly.

\begin{thm}[Theorem 3.21 \cite{bmov13}]\label{t:sibliclos}
The family of all invariant geodesic laminations $\lam$ is compact in
$\mc(\mc(\cdisk))$.
\end{thm}

In other words, if invariant geodesic laminations $\lam_i$ converge to a collection of chords $\lam$ in $\mc(\mc(\cdisk))$ (that is, each leaf of $\lam$ is the limit of a sequence of leaves from $\lam_i$, and each converging sequence of leaves of $\lam_i$ converges to a leaf of $\lam$), then $\lam$ is an invariant geodesic lamination itself.

\subsection{Laminational equivalence relations}\label{sss:laeqre}
Geodesic laminations naturally appear in the context of laminational equivalence relations.

\begin{dfn}[Laminational equivalence relations]\label{d:lam}
An equi\-va\-lence re\-la\-tion $\sim$ on the unit circle $\uc$ is said to be \emph{laminational} if either $\uc$ is one
equivalence class (such laminational equivalence relations are called \emph{degenerate}), or the following holds:

\noindent (E1) the graph of $\sim$ is a closed subset of $\uc \times \uc$;

\noindent (E2) the convex hulls of distinct equivalence classes are disjoint;

\noindent (E3) each equivalence class of $\sim$ is finite.

\label{d:si-inv-lam}
A laminational equivalence relation $\sim$ is called ($\si_d$-){\em invariant} if:

\noindent (D1)
$\sim$ is {\em forward invariant}: for a $\sim$-class $\g$, the set $\si_d(\g)$ is a $\sim$-class;

\noindent (D2)
for any $\sim$-equivalence class $\g$, the map $\si_d: \g\to\si_d(\g)$ extends to $\uc$ as an orientation preserving covering
map such that $\g$ is the full preimage of $\si_d(\g)$ under this covering map.
\end{dfn}

For an invariant laminational equivalence relation $\sim$, consider the \emph{topological Julia set} $J_\sim=\uc/\hspace{-5pt}\sim$ and the \emph{topological polynomial} $f_\sim:J_\sim\to J_\sim$ induced by $\si_d$.
The quotient map $\pi_\sim:\uc\to J_\sim$ semiconjugates $\si_d$ with $f_\sim|_{J_\sim}$.
A laminational equivalence relation $\sim$ admits a \textbf{canonical extension over $\C$} whose non-trivial classes are convex hulls of classes of $\sim$.
By Moore's Theorem, the quotient space $\C/\hspace{-5pt}\sim$ is homeomorphic to $\C$.
We will still denote the extended quotient map from $\C$ to $\C/\hspace{-5pt}\sim$ by $\pi_\sim$;
the corresponding point-preimages (\emph{fibers}) are the convex hulls of 
$\sim$-classes.
With any fixed identification between $\C/\sim$ and $\C$, one can extend $f_\sim$ to a branched covering map $f_\sim:\C\to \C$ of degree $d$ called a \emph{topological polynomial} too.
The complement $K_\sim$ of the unique unbounded component $U_\infty(J_\sim)$ of $\C\sm J_\sim$ is called the \emph{filled topological Julia set}.
Define the \emph{canonical geodesic lamination $\lam_\sim$ generated by $\sim$} as the collection of edges of convex hulls of all $\sim$-classes and all points of $\uc$.

\begin{lem}[Theorem 3.21 \cite{bmov13}]\label{t:qsib}
Geodesic laminations $\lam_\sim$ generated by $\si_d$-invariant laminational equivalence relations are invariant.
\end{lem}

\subsection{Dendritic case}\label{ss:dend-cas}
We now 
consider dendritic laminations and corresponding topological polynomials.

\begin{dfn}\label{d:dendri-all}
An invariant geodesic lamination $\lam_\sim$ is called \emph{dendritic} if all its gaps are finite.
Then the corresponding topological Julia set $\uc/\sim$ is a dendrite.
The laminational equivalence relation $\sim$ and the topological polynomial $f_\sim$ are said to be \emph{dendritic} too.
\end{dfn}

Recall that, by \cite{kiwi97}, with every dendritic polynomial $P$ one can associate a dendritic topological polynomial
$f_{\sim_P}$ so that $P|_{J(P)}$ is monotonically semi-conjugate to $f_{\sim_P}|_{J(f_{\sim_P})}$.
By \cite{kiw05}, for every dendritic topological polynomial $f$, there exists a polynomial $P$ with $f=f_{\sim_P}$.
Below, we list some well-known properties of dendritic geodesic laminations.

\begin{dfn}[Perfect parts of geodesic laminations \cite{bopt16}]\label{perfect}
Let $\lam$ be a geodesic lamination considered as a subset of $\mc(\cdisk)$.
Then the maximal perfect subset $\lam^p$ of $\lam$ is called the \emph{perfect part} of $\lam$.
A geodesic lamination $\lam$ is called \emph{perfect} if $\lam=\lam^p$.
Equivalently, this means that all leaves of $\lam$ are non-isolated in the Hausdorff metric.
\end{dfn}

Observe that $\lam^p$ must contain $\uc$.
The following lemma is well-known.

\begin{lem}\label{l:perfectd}
Dendritic geodesic laminations $\lam$ are perfect.
\end{lem}

We will need Corollary 6.6 of \cite{bopt16}, which reads:

\begin{cor}\label{c:cridisj}
Let $\lam$ be a perfect invariant geodesic lamination.
Then the critical objects of $\lam$ are pairwise disjoint and are either all-critical sets, or critical sets whose boundaries map exactly $k$-to-$1$, $k>1$, onto their images.
\end{cor}

By \Cref{l:perfectd}, \Cref{c:cridisj} applies to dendritic geodesic laminations.
Moreover, by properties of dendritic geodesic laminations, all their critical objects are finite.

Finally, let us state a property that follows from \Cref{d:dendri-all}.
If $\lam$ is a dendritic geodesic lamination, and $\lam'\supset \lam$ is an invariant geodesic lamination,
then it follows that $\lam'$ can be obtained from $\lam$ by inserting leaves in some grand orbits of gaps of $\lam$. Moreover, as long as this insertion is done in a dynamically consistent fashion (added leaves do not cross and form a fully invariant set), the new collection of leaves $\lam'$ is an invariant geodesic lamination.
On the other hand, the fact that $\lam'$ is closed follows from the fact that if gaps of $\lam$ converge, then they must
converge to a leaf of $\lam$.

\section{Linked quadratically critical geodesic laminations}\label{s:long}

Now we will review results of \cite{bopt16} that are essential for this paper.
Let us emphasize that results of \cite{bopt16} hold for any degree.
However, we will adapt them here to degree three, omitting the general formulations.

Consider a quadratic lamination $\lam$ with a critical quadrilateral $Q$.
Thurston \cite{thu85} associates to $\lam$ its minor $\mf=\si_2(Q)$.
Then $Q\cap\uc$ is the full $\si_2$-preimage of $\mf\cap\uc$.
Thurston proves that different minors obtained in this way never cross in $\disk$.
Observe that two minors cross if and only if the vertices of the corresponding critical quadrilaterals strictly alternate in $\uc$.
Thurston's result can be translated as follows in terms of critical quadrilaterals.
If two quadratic laminations generated by laminational equivalences have critical quadrilaterals whose vertices
strictly alternate, then the two laminations are the same.
This motivates Definition~\ref{d:strolin}.

\begin{dfn}\label{d:strolin}
Let $A$ and $B$ be two quadrilaterals.
Say that $A$ and $B$ are \emph{strongly linked} if the vertices of $A$ and $B$ can be numbered so that
$$
a_0\le b_0\le a_1\le b_1\le a_2\le b_2\le a_3\le b_3\le a_0,
$$
where $a_i$, $0\le i\le 3$, are vertices of $A$ and $b_i$, $0\le i\le 3$ are vertices of $B$.
\end{dfn}

Since we want to study critical quadrilaterals in the degree three case, we now give a general definition.

\begin{dfn}\label{d:qll}
A \emph{(generalized) critical quadrilateral} $Q$ is a circularly ordered quadruple $[a_0,a_1,a_2,a_3]$ of points $a_0\le a_1\le a_2\le a_3\le a_0$ in $\uc$, where $\ol{a_0a_2}$ and $\ol{a_1a_3}$ are critical chords called \emph{spikes};
critical quadrilaterals $[a_0,a_1,a_2,a_3]$, $[a_1,a_2,a_3,a_0]$, $[a_2,a_3,a_0,a_1]$ and $[a_3,a_0,a_1,a_2]$ are viewed as equal.
\end{dfn}

We will often say ``critical quadrilateral'' when talking about the convex hull of a critical quadrilateral.
Clearly, if all vertices of a critical quadrilateral are distinct, or if its convex hull is a critical leaf, then the quadrilateral is uniquely defined by its convex hull.
However, if the convex hull is a triangle, this is no longer true.
For example, let $\ch(a, b, c)$ be an all-critical triangle.
Then $[a,a,b,c]$ is a critical quadrilateral, but so are $[a,b,b,c]$ and $[a,b,c,c]$.
If all vertices of a critical quadrilateral $Q$ are pairwise distinct, then we call $Q$ \emph{non-degenerate}.
Otherwise $Q$ is called \emph{degenerate}.
Vertices $a_0$ and $a_2$ ($a_1$ and $a_3$) are called \emph{opposite}.

Considering invariant geodesic laminations, all of whose critical sets are critical quadrilaterals, is not very restrictive:
we can ``tune'' a given invariant geodesic lamination by inserting new leaves into its critical sets in order to construct a new invariant geodesic lamination with all critical objects being critical quadrilaterals.

\begin{lem}[Lemma 5.2 \cite{bopt16}]\label{l:qls}
The family of all critical quadrilaterals is closed.
The family of all critical quadrilaterals that are critical sets of invariant geodesic laminations is closed too.
\end{lem}

Being strongly linked is a closed condition on two quadrilaterals: if two sequences of critical quadrilaterals $A_i$, $B_i$ are such that $A_i$ and $B_i$ are strongly linked and $A_i\to A$, $B_i\to B$, then $A$ and $B$ are strongly linked critical quadrilaterals.

In \cite{bopt16}, linked invariant geodesic laminations with quadratically critical portraits are defined
for any degree $d$.
Below, we 
adapt this in the case of cubic laminations.
Let $\lam$ be a cubic geodesic lamination; in fact, in what follows when talking about
laminations we always mean \emph{cubic} laminations (unless explicitly stated otherwise).
Consider critical quadrilaterals $Q^1$, $Q^2$ that are leaves or gaps of $\lam$.
The pair $(Q^1,Q^2)$ is called a \emph{quadratically critical portrait} of $\lam$ if $Q^1$ and $Q^2$
are distinct. 
The triple $(\lam,Q^1,Q^2)$ is then called a \emph{cubic lamination with quadratically critical portrait}.
Sometimes, when the quadratically critical portrait is fixed, we write $\lam$ instead of $(\lam,Q^1,Q^2)$.
Observe that not all cubic geodesic laminations admit quadratically critical portraits.
For example, if $\lam$ has a unique critical object that is not all-critical, then $\lam$ has no quadratically critical portrait.
If $\lam$ has two disjoint critical objects, then $Q^1$, $Q^2$ must coincide with these objects.
In particular, a cubic geodesic lamination with two disjoint critical objects admits a quadratically critical portrait if and only if both critical objects are (possibly degenerate) critical quadrilaterals.

Assume that $\lam$ has an all-critical triangle $\Delta$.
Then possible quadratically critical portraits of $\lam$ are:
\begin{enumerate}
 \item pairs of distinct edges of $\Delta$; and
 \item pairs consisting of $\Delta$ and an edge of it.
\end{enumerate}

Now we define linked laminations with quadratically critical portraits.

\begin{dfn}\label{d:qclink1}
Let $(\lam_1,Q_1^1,Q_1^2)$ and $(\lam_2,Q_2^1,Q_2^2)$ be cubic invariant geodesic la\-mi\-na\-tions with quadratically critical portraits.
These two laminations are said to be \emph{linked or essentially equal} if one of the following holds:
\begin{enumerate}
 \item
 For every $j=1,2$, the quadrilaterals $Q_1^j$ and $Q_2^j$ are strongly linked.
 If $Q_1^j$ and $Q_2^j$ share a spike for every $j=1,2$, then $\lam_1$ and $\lam_2$ are said to be essentially equal.
 \item
 The laminations $\lam_1$ and $\lam_2$ share the same all-critical triangle.
 Then $\lam_1$ and $\lam_2$ are also said to be essentially equal.
\end{enumerate}
If (1) or (2) holds but $\lam_1$, $\lam_2$ are not essentially equal, then $\lam_1$, $\lam_2$ are said to be \emph{linked}.
We can also talk about linked or essentially equal quadratically critical portraits without referring to laminations containing them.
\end{dfn}

Critically marked polynomials, topological polynomials, and
laminational equivalence relations were defined in the introduction;
recall that there they are all assumed to be dendritic.
Let us now define critically marked cubic \emph{geodesic} laminations.
Suppose that $\lam$ is a cubic geodesic lamination and an ordered pair
of critical sets (gaps or leaves) $C, D$ of $\lam$ is chosen so that on
the boundary of each component $E$ of $\cdisk\sm (C\cup D)$ the map
$\si_3$ is one-to-one (except for the endpoints of possibly existing
critical edges of such components). Then we can consider $(\lam, C, D)$
as a \emph{(critically) marked} lamination even though $\lam$ is not
necessarily dendritic. For brevity we often talk about \emph{marked}
(topological) polynomials and laminations rather than \emph{critically
marked} ones. Let $(\lam,C^1,C^2)$ be a marked cubic geodesic
lamination. Then $(C^1,C^2)$ is called a \emph{critical pattern} of
$\lam$; when talking about \emph{critical patterns} we mean critical
patterns of some marked lamination $\lam$ and allow for $\lam$ to be
unspecified.

Let $\lam$ be a dendritic lamination. Then if $C\ne D$ are its critical
sets, the only two possible critical patterns that can be associated
with $\lam$ are $(C, D)$ or $(D, C)$. Now, if $\lam$ has a unique
critical set $X$ which is not an all-critical triangle then the only
possible critical pattern of $\lam$ is $(X, X)$. However if $\lam$
has a unique critical set $\Delta$ which is an all-critical triangle then
there are more possibilities for a critical pattern of $\lam$. Namely,
by definition a critical pattern of $\lam$ can be either $(\Delta, \Delta)$,
of $\Delta$ and an edge of $\Delta$, or an edge of $\Delta$ and $\Delta$, or
an ordered pair of two edges of $\Delta$.

We defined linked or essentially equal cubic laminations with quadratically critical portraits.
Let us extend this notion to marked laminations and their critical patterns. A \emph{collapsing quadrilateral} is a critical
quadrilateral that maps to a non-degenerate leaf.

\begin{dfn}\label{d:link3}
Marked laminations $(\lam_1,C_1^1,C_1^2)$ and $(\lam_2,C_2^1,C_2^2)$,
and their critical patterns $(C_1^1,C_1^2)$ and $(C_2^1,C_2^2)$, are
said to be \emph{linked} (\emph{essentially equal}) if there are
quadratically critical portraits $(Q_1^1,$ $Q_1^2)$ and $(Q_2^1,Q_2^2)$
such that $Q_i^j\subset C_i^j,\quad \lamm_i\supset\lam_i,\quad
i,j=1,2,$ for which if $Q_i^j$ is a collapsing quadrilateral, then it
shares a pair of opposite edges with $C_i^j$.
\end{dfn}

Let $Q$ be one of $Q_i^j$, and $C$ be the corresponding $C_i^j$.
Then the image of $Q$ is an edge of $\si_3(C)$. 
Hence forward images of quadrilaterals $Q_i^j$ will never form linked leaves.
Pulling back the sets $Q_i^j$, we construct the ``tuned'' laminations $\lamm_1$ and $\lamm_2$
with quadratically critical portraits $Q_1^1,$ $Q_1^2)$ and $(Q_2^1,Q_2^2)$. The existence
of ``tuned'' laminations with quadratically critical portraits follows
from our definitions and Thurston's pullback construction.

Lemma~\ref{l:actri} follows immediately from definitions.

\begin{lem}\label{l:actri}
If two limit laminations contain the same all-critical
triangle, then they are essentially equal.
\end{lem}

The following is a special case of one the central results of \cite{bopt16}
(in \cite{bopt16} more general results are obtained for all degrees).

\begin{thm}[Theorem 9.6 \cite{bopt16}]\label{t:noesli}
Let $(\lam_1,C_1^1,C_1^2)$ and $(\lam_2,C_2^1,C_2^2)$ be marked laminations.
Suppose that $\lam_1$ is perfect (e.g., by Lemma~\ref{l:perfectd} $\lam_1$ may be dendritic).
If they are linked or essentially equal then $\lam_1\subset\lam_2$ and $C_1^j\supset C_2^j$ for $j=1,2$.
In particular, if both laminations are perfect, then $(\lam_1,C_1^1,C_1^2)=(\lam_2,C_2^1,C_2^2)$.
\end{thm}

\section{Proof of the main result}\label{s:cudebi}
In the rest of the paper, we define a visual parameterization of the family 
of all marked cubic dendritic geodesic laminations.

\subsection{Limit laminations}\label{ss:ll}
Let $(\lam_i, \zc_i)$ be a sequence of marked laminations with critical
patterns $\zc_i=(C_i^1, C^2_i)$ where $C_i^1\cap \uc, C_i^2\cap \uc$
are finite. Assume that the sequence $\lam_i$ converges to an invariant
lamination $\lam_\infty$ (see Theorem~\ref{t:sibliclos}); then the
critical sets $C_i^1$, $C_i^2$ converge to gaps (or leaves)
$C_\infty^1, C_\infty^2$ of $\lam_\infty$. We say that the
sequence $(\lam_i, \zc_i)$ \emph{converges} to $(\lam_\infty,
C_\infty^1, C_\infty^2)$.

\begin{lem}\label{l:decrilim}
Let a sequence $(\lam_i, \zc_i)$ of marked laminations with finite
critical sets converge to $(\lam_\infty, C_\infty^1, C_\infty^2)$. Then
sets $C_\infty^1, C_\infty^2$ are critical and non-periodic, and
$(\lam_\infty, C_\infty^1, C_\infty^2)$ is a marked lamination.
\end{lem}

\begin{proof}
Every vertex of $C_\infty^1$ has a sibling vertex in $C_\infty^1$. It
follows that $C_\infty^1$ is critical. If $C_\infty^1$ is periodic of
period, say, $n$, then, since it is critical, it is an infinite gap.
Then the fact that $\si_d^n(C_\infty^1)=C_\infty^1$ implies that any
gap $C_i^1$ sufficiently close to $C_\infty^1$ will have its
$\si_3^n$-image also close to $C_\infty^1$, and therefore coinciding with $C_i^1$.
Thus, $C_i^1$ is $\si_3$-periodic, which is impossible because $C_i^1$ is
finite and critical. Similarly, $C_\infty^2$ is critical and non-periodic.

Let us show that $(\lam_\infty, C_\infty^1, C_\infty^2)$ is a marked
lamination. To this end we need to show that on the boundary of
each component $E$ of $\cdisk\sm (C_\infty^1\cup C_\infty^2)$ the map
$\si_3$ is one-to-one (except for the endpoints of possibly existing
critical edges of such components). However this follows from
definitions and the fact that the same claim holds for all $(\lam_i,
\zc_i)$.
\end{proof}

Any marked lamination similar to $(\lam_\infty, C_\infty^1,
C_\infty^2)$ from Lemma~\ref{l:decrilim} will be called a \emph{limit
marked lamination}. In particular, a marked dendritic lamination is a
limit marked lamination (consider a constant sequence). Recall that for a compact
metric space $X$, the space of all its compact subsets with the
Hausdorff metric is denoted by $\mc(X)$.

As was explained in the Introduction, a marked cubic dendritic
polynomial always defines a marked cubic lamination. Take a marked
dendritic polynomial $(P, c^1, c^2)$ and let $(\lam,C^1,C^2)$ be the corresponding
marked lamination. 
Define the map $\Gamma:\md_3\to \mc(\cdisk)\times \mc(\cdisk)$ by setting
$\Gamma(P, c^1, c^2)=(C^1, C^2)$. Consider a sequence of marked dendritic
cubic geodesic laminations $(\lam_i,C_i^1,C_i^2)$. If $\lam_i$ converge
then, by Theorem~\ref{t:sibliclos}, the limit $\lam_\infty$ is itself a
cubic geodesic lamination, and by the above the critical patterns
$(C_i^1,C_i^2)$ converge to the critical pattern $(C_\infty^1,
C_\infty^2)$ of $\lam_\infty$. We are interested in the case when
$\lam^\infty$ is in a sense compatible with a dendritic lamination.

\begin{lem}[Lemma 6.18 \cite{bopt16}]\label{l:uppers}
Let $(\lam_i,C_i^1,C_i^2)$ and $(\lam_\infty,C_\infty^1,C_\infty^2)$ be
as above. If there exists a dendritic cubic geodesic lamination $\lam$
with a critical pattern $(C^1,C^2)$ such that $C^j_\infty\subset C^j$
for $j=1,2$. Then $\lam_\infty\supset \lam$.
\end{lem}

Lemma~\ref{l:uppers} says that if critical patterns of dendritic cubic geodesic laminations converge \textbf{into}
a critical pattern of a dendritic cubic geodesic lamination $\lam$, then the limit lamination contains $\lam$.

\begin{cor}[Corollary 6.20 \cite{bopt16}]\label{c:crista}
Suppose that a sequence $(P_i, c_i^1,$ $c_i^2)$ of marked cubic dendritic
polynomials converges to a marked cubic dendritic polynomial $(P, c^1,
c^2)$. Consider corresponding marked laminational equivalence relations
$(\sim_{P_i}, C_i^1, C_i^2)$ and $(\sim_P, C^1, C^2)$. If
$(\lam_{\sim_{P_i}},C_i^1,C_i^2)$ converges to $(\lam_\infty,$
$C_\infty^1,$ $C_\infty^2)$, then we have $ \lam_\infty\supset
\lam_{\sim_P}, C_\infty^1\subset C^1, C_\infty^2\subset C^2.$ In particular,
the map $\Gamma$ is upper semi-continuous.
\end{cor}

By Corollary~\ref{c:crista}, critical objects of dendritic invariant geodesic laminations $\lam_{\sim_P}$
associated with polynomials $P\in\md_3$ cannot explode under perturbation of $P$ (they may implode though).

\subsection{Mixed tags of geodesic laminations}

\begin{dfn}[Minor set]\label{d:minor}
Let $(\lam, C, D)$ be a marked lamination. Then $\si_3(D)$ is called the
\emph{minor set of $(\lam,C, D)$}.
\end{dfn}

Note that, in Definition \ref{d:admicri}, the set $C$ is not assumed to be critical.

\begin{dfn}[Co-critical set]\label{d:admicri}
Let $C$ be a leaf or a gap of a cubic invariant geodesic lamination $\lam$.
Assume that either $C$ is the only critical object of $\lam$, or there is exactly one hole of $C$ of length $>\frac13$.
If $C$ is the only critical object of $\lam$, we set $\coc(C)=C$.
Otherwise let $H$ be the unique hole of $C$ of length $>\frac 13$,
let $A$ be the set of all points in 
$H$ with the images in $\si_3(C)$, and
set $\coc(C)=\ch(A)$.
The set $\coc(C)$ is called the \emph{co-critical set} of $C$.
\end{dfn}

We now define tags of marked laminations.

\begin{dfn}[Mixed tag]\label{d:siblita}
Suppose that $(\lam,C^1,C^2)$ is a marked lamination.
Then we call the set $\ta_l(C^1,C^2)=\coc(C^1)\times \si_3(C^2)\subset \cdisk\times \cdisk$ the \emph{mixed tag} of $(\lam,C^1,C^2)$ or of $(C^1,C^2)$.
\end{dfn}

Sets $\coc(C^1)$ (and hence mixed tags) are well-defined. The mixed tag
$T$ of a marked lamination is the product of two sets, each of which is
a point, a leaf, or a gap. One can think of $T\subset \cdisk\times
\cdisk$ as a higher dimensional analog of a gap/leaf of a geodesic
lamination. We show that the union of tags of marked dendritic
laminations is a (non-closed) ``geodesic lamination'' in
$\ol{\disk}\times\ol{\disk}$. The main idea is to relate the
non-disjointness of mixed tags of marked dendritic laminations and
their limits with the fact that they have ``tunings'' that are linked
or essentially equal.

In Definition~\ref{d:unibi}, we mimic Milnor's terminology for
polynomials.

\begin{dfn}[Unicritical and bicritical laminations]\label{d:unibi}
A marked la\-mi\-nation (and its critical pattern) is called
\emph{unicritical} if its critical pattern is of form $(C, C)$ for some
critical set $C$ and \emph{bicritical} otherwise.
\end{dfn}

Clearly, a unicritical marked lamination has a unique critical object.
However a lamination $\lam$ with unique critical object may have a
bicritical critical pattern. By definition this is only possible if
$\lam$ has an all-critical gap $\Delta$ and the critical pattern is
either two edges of $\Delta$ or $\Delta$ and an edge of $\Delta$.

The following lemma is a key combinatorial fact about tags.

\begin{lem}\label{l:interlink}
Suppose that two marked laminations have non-disjoint mixed tags.
Suppose also that at least one of the two laminations is dendritic.
Then either the two marked laminations are linked or essentially equal, or
the geodesic laminations are equal and share an all-critical triangle.
\end{lem}

The proof of Lemma \ref{l:interlink} is mostly non-dynamic and
involves checking various cases. We split the proof into
propositions. Observe that mixed tags are determined by critical
patterns; we do not need laminations to define mixed tags. In
Propositions~\ref{p:mnondeg} - \ref{p:quadfup} we assume that the
critical patterns $(C_1^1, C_1^2)$ and $(C_2^1,
C_2^2)$ are bicritical and have non-disjoint mixed tags.

\begin{prop}
 \label{p:mnondeg} Suppose that some distinct edges of $\coc(C_1^1)$
 and $\coc(C_2^1)$ cross. Then the two critical patterns are linked or
 essentially equal.
\end{prop}

\begin{proof}
By the assumption, some distinct edges of the sets $\coc(C_1^1)$ and
$\coc(C_2^1)$ cross. Denote these linked edges by $\ol{a_1b_1}$ and
$\ol{a_2b_2}$, see Fig. \ref{f:mnondeg}. We may choose the orientation
so that $(a_1, b_1)$, $(a_2, b_2)$ are in the holes of $\coc(C_1^1)$,
$\coc(C_2^1)$ disjoint from $C_1^1$, $C_2^1$ respectively, and so that
$a_1< a_2< b_1< b_2$. We claim that $(a_1, b_1)$ is of length at most
$\frac13$. Indeed, if $(a_1, b_1)$ had length greater than $\frac13$,
then there would exist a sibling $\ell$ of $\ol{a_1b_1}$ with endpoints
in $(a_1, b_1)$. Evidently, $\ell$ would be an edge of $C^1_1$,
contradicting the choice of $(a_1, b_1)$. Thus, $(a_1, b_1)$ is of
length at most $\frac13$ and the restriction $\si_3|_{(a_1, b_1)}$ is
one-to-one. Similarly, $(a_2, b_2)$ is of length at most $\frac13$ and
the restriction $\si_3|_{(a_2, b_2)}$ is one-to-one.

Let us show now that $\si_3(C_1^2)\cap\uc\subset [\si_3(b_1), \si_3(a_1)]$.
If $C^1_1=C_1^2$ is of degree three, this follows immediately.
Otherwise, let $a_1'=a_1+\frac13$ and $b_1'=b_1+\frac23$.
Then $\ol{a_1'b_1'}\subset C^1_1$. Moreover, since $C^1_1$ is critical,
vertices of $C^1_1$ partition the arc $(a_1', b_1')$ into open arcs on each
of which the map is one-to-one. Hence $C_1^2$ must have
vertices in $[b_1', a_1]\cup[b_1, a_1']$. Since each of these intervals maps onto
$[\si_3(b_1),\si_3(a_1)]$ one-to-one, our claim follows.

We claim that $b_2\le a_1+\frac13$. Indeed, otherwise
$[b_1,a_1+\frac13]\subset [a_2,b_2)$, which implies that
$[\si_3(b_1),\si_3(a_1)]\subset [\si_3(a_2), \si_3(b_2))$. On the other
hand, by the above we have $\si_3(C_1^1)\subset
[\si_3(b_1),\si_3(a_1)]$ and $\si_3(C_2^1)\subset
[\si_3(b_2),\si_3(a_2)]$. Since $\si_3(C_1^2)\cap \si_3(C_2^2)\ne \0$,
we have in fact $b_1=a_2$, a contradiction. Thus, the points $a_i$
and $b_i$ for $i=1, 2$ belong to an arc of length at most $\frac13$.

We claim that then $\coc(\ol{a_1b_1})=Q^1_1$ and
$\coc(\ol{a_2b_2})=Q_2^1$ are strongly linked collapsing
quadrilaterals. Indeed, we have that $a_1<a_2<b_1<b_2\le a_1+\frac13$.
It follows then that
$$a_1+\frac13<a_2+\frac13<b_1+\frac13<b_2+\frac13\le
a_1+\frac23<a_2+\frac23<b_1+\frac23<b_2+\frac23\le a_1,$$ 
and therefore that, indeed, $Q^1_1$ and
$\coc(\ol{a_2b_2})=Q^1_2$ are strongly linked collapsing
quadrilaterals. Moreover, since $\ol{a_1b_1}$ and $\ol{a_2b_2}$ are
edges of $\coc(C^1_1)$ and $\coc(C_2^1)$ it follows that the quadrilateral
$Q^1_1$ shares two edges with the set $C^1_1$, and the quadrilateral
$Q^1_2$ shares two edges with the set $C_2^1$.

Note that all vertices of $C_1^2$ and $C_2^2$ are in $[b_2,a'_1]\cup
[b'_2,a_1]$, where $a'_1=a_1+\frac13$ and $b'_2=b_2+\frac 23$. The
restriction of $\si_3$ to each of the arcs $[b_2,a'_1]$, $[b'_2,a_1]$
is injective. Therefore, a pair of linked edges (or a pair of
coinciding vertices) of $\si_3(C_1^2)$ and $\si_3(C_2^2)$ gives rise to
a pair of linked quadrilaterals $Q_1^2$ and $Q_2^2$ in $C_1^2$ and
$C_2^2$, respectively, so that if these quadrilaterals share edges with containing
them critical sets.
\end{proof}

\begin{figure}
 \includegraphics[width=6cm]{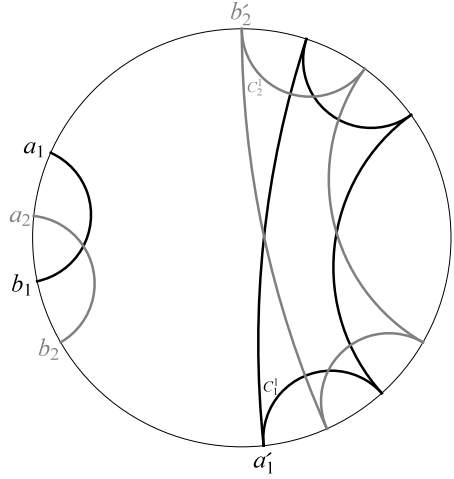}
 \caption{This figure illustrates Proposition~\ref{p:mnondeg}.}
 \label{f:mnondeg}
\end{figure}

Now we simply assume that $\coc(C_1^1)$ and $\coc(C_2^1)$ intersect. %

\begin{prop}
 \label{p:quadfup}
 If $\lam_1$ is dendritic, and $\lam_2$ is a limit marked lamination,
then at least one of the following holds:
 \begin{enumerate}
 \item  the two critical patterns are linked or essentially equal;
 \item  $\lam_1=\lam_2$ share an all-critical triangle $\Delta$. 
 \end{enumerate}
\end{prop}

\begin{proof}
We will use the same notation as in the proof of Proposition
\ref{p:mnondeg}. If $\coc(C_1^1)$ and $\coc(C_2^1)$ have distinct edges
that cross in $\disk$, then Proposition \ref{p:mnondeg} applies. Assume
now that $\coc(C_1^1)$ and $\coc(C_2^1)$ share a vertex $a$. Clearly,
there is a unique critical chord $\ell$ such that $\coc(a)=\ell$. Then
$C_1^1\cap C^1_2\supset \ell$, and we may set
$Q_1^1=Q_2^1=\ell$.

Both sets $C_i^2$ have vertices in the closed arc $A$ of length $\frac
23$ bounded by the endpoints of $\ell$. By our assumption,
$\si_3(C_1^2)\cap\si_3(C_2^2)\ne\0$. If the sets $\si_3(C_i^2), i=1,2$ have a pair of
linked edges or share a vertex $z\ne \si_3(\ell)$, then these edges or
$z$ can by pulled back to $\ch(A)$ as a pair of linked critical
quadrilaterals. Assume now that $\si_3(C_1^2)\cap
\si_3(C_2^2)=\{\si_3(\ell)\}$.

Clearly, $a\in A$. Set $\Delta=\ch(a,\ell)$. We claim that $\Delta$ is
a gap of $\lam_1$. Indeed, the set $C_1^2$ contains at least two
vertices of $\Delta$ and is non-disjoint from $C_1^1$. Since
$\lam_1$ is dendritic, $\lam_1$ has a unique critical object $E$.
If $E\ne \Delta$, then by definition the critical pattern of $\lam_1$
is $(E, E)$, a contradiction with the assumption that $\lam_1$ is
bicritical. Thus, $\Delta$ is a gap of $\lam_1$.

We claim that $\Delta$ is a gap $\lam_2$.
We prove first that there is an edge $\ell^*\ne \ell$ of $\Delta$ such
that one of the sets $C_2^1, C_2^2$ contains $\ell$ while the other one
contains $\ell^*$. This is obvious if $C_2^2$ contains an edge $\ell^*\ne \ell$ of $\Delta$.
Otherwise $C_2^2\supset \ell$. Then
$\ell$ must be an edge of $C_2^2$ because otherwise the sets $C_2^1$
and $C_2^2$ will have either non-disjoint interiors, or one of them is
contained in the interior of the other one, a contradiction. Similarly,
$\ell$ is an edge of $C_2^1$. It follows that one of the sets $C_2^1,
C_2^2$ is $\ell$ while the other one is a critical gap $G\ne \Delta$ with $\ell$ as an
edge.

By the above, $\ell$ and $\ell^*$ are either leaves of $\lam_2$ or are
contained in gaps of $\lam_2$. Moreover endpoints of $\ell$ and
$\ell^*$ are not periodic since $\Delta$ is a gap of a dendritic
lamination $\lam_1$. Hence $\ell$ and $\ell^*$ can be pulled back in a
unique way and its pullbacks either will be contained in gaps of
$\lam_2$ or will be leaves of $\lam_2$. This yields a new lamination
$\hlam_2\supset \lam_2$ and a marked lamination $(\hlam_2, \ell,
\ell^*$). Consider also the marked lamination $(\lam_1, \ell, \ell^*)$.
Since these two marked laminations are essentially equal,
Theorem~\ref{t:noesli} implies that $\lam_1\subset \hlam_2$. Hence
$\Delta$ is a gap of $\hlam_2$ and, moreover, leaves shared by $\lam_1$
and $\hlam_2$ approximate all edges of $\Delta$ from outside of
$\Delta$.

It follows that $\Delta$ is a subset of a gap $G$ of $\lam_2$. Let us
show that $G=\Delta$. By Lemma \ref{l:decrilim}, $G$ is not periodic.
Hence pullbacks of $\ell$ and $\ell^*$ do not re-enter $G$, and so an
edge of $\Delta$ contained in the interior of $G$ (except for the
endpoints) remains isolated in both $\lam_2$ and $\hlam_2$. However in
the previous paragraph we concluded that it is not isolated in
$\hlam_2$, a contradiction. We conclude that $\Delta$ is a gap of
$\lam_2$.

Let us show that $\lam_1=\lam_2$.
We can adjust the critical pattern of $\lam_2$ so that it coincides with the critical pattern of $\lam_1$.
By Theorem~\ref{t:noesli}, we then have $\lam_2\supset\lam_1$.
Moreover, no leaves of $\lam_2$ are contained in the unique critical set $\Delta$ of $\lam_1$.
By \cite{kiw02}, any periodic gap of $\lam_1$ has a single cycle of edges.
We conclude that no leaves of $\lam_2$ are contained in periodic or preperiodic gaps of $\lam_1$.
Finally, by \cite{bl02} there are no wandering gaps of $\lam_1$.
This implies that $\lam_2=\lam_1$, as claimed.
\end{proof}

This proves Lemma \ref{l:interlink} for two bicritical marked laminations.
Consider unicritical marked laminations.

\begin{lem}\label{l:uniclink}
Suppose that $(\lam_1,C_1,C_1)$ and $(\lam_2,C_2,C_2)$ are
marked unicritical laminations with non-disjoint mixed tags. Then
$(\lam_1,C_1^1,C_1^2)$ and $(\lam_2, C_2^1,C_2^2)$ are linked or
essentially equal where $C_i^j$ either equals $C_i$ or is a critical chord
contained in $C_i$.
\end{lem}

\begin{proof}
Suppose that $\lam_1$ has an all-critical triangle $\Delta$ (and so
$C_1=\Delta$). Since the mixed tags intersect, then $\si_3(C_1)\in
\si_3(C_2)$ and hence $C_1\subset C_2$. Choosing two edges of $\Delta$
as a quadratically critical portrait in $C_1$ and in $C_2$, we see that
by definition $(\lam_1,C_1^1,C_1^2)$ and $(\lam_2, C_2^1,C_2^2)$ are
essentially equal. Suppose that neither invariant geodesic
lamination has an all-critical triangle. If $\si_3(C_1)\cap \si_3(C_2)$
contains a point $x\in \uc$, then the entire all-critical triangle
$\ch(\si_3^{-1}(x))=\Delta$ is contained in $C_1\cap C_2$; we can
choose the same two edges of $\Delta$ as a quadratically critical
portrait for both laminations. Otherwise, we may assume that an edge
$\ell_1$ of $\si_3(C_1)$ crosses an edge $\ell_2$ of $\si_3(C_2)$. This
implies that the hexagons $\si_3^{-1}(\ell_1)\subset C_1$ and
$\si_3^{-1}(\ell_2)\subset C_2$ have alternating vertices and proves
the lemma in this case too.
\end{proof}


\begin{proof}[Proof of Lemma \ref{l:interlink}]
Denote laminations in question by $\lam_1$ and $\lam_2$. If both
laminations are bicritical, then the result follows from Proposition
\ref{p:quadfup}. If both laminations are unicritical, then the result
follows from Lemma \ref{l:uniclink}. It remains to consider the case
where the first critical pattern $(C_1, C_1)$ is unicritical, and the
second one $(C_2^1, C_2^2)$ is bicritical.


Either an edge of $\si_3(C_1)$ crosses an edge of $\si_3(C_2^2)$ or a vertex of $\si_3(C_1)$ lies in $\si_3(C_2^2)$.
In either case, there are two sibling edges or sibling vertices of $C_2^2$ that are linked or coincide with edges or vertices of $C_1$.
Taking convex hulls of these pairs of (possibly degenerate) leaves, we obtain $Q_1^2\subset C_1$ and $Q_2^2\subset C_2^2$.
By construction, these are strongly linked quadrilaterals.
Similarly, there is a (possibly degenerate) leaf in $\coc(C_2^1)$ that is linked or equal to a leaf in $C_1$.
It follows that the two siblings of this leaf in $C_2^1$ are linked or equal to some leaves in $C_1$.
As above, this leads to strongly linked quadrilaterals $Q_1^1\subset C_1$ and $Q_2^1\subset C_2^1$.
It is easy to see that $(Q_1^1,Q_1^2)$ and $(Q_2^1,Q_2^2)$ are linked or essentially equal quadratically critical portraits.
\end{proof}

We are ready to prove Theorem \ref{t:dendrilink}. 

\begin{thm}\label{t:dendrilink}
If $(\lam_1,C_1^1,C_1^2)$ and $(\lam_2,C_2^1,C_2^2)$ are marked laminations and
$\lam_1$ is dendritic, then they have non-disjoint mixed tags if and only if $(1)$ or $(2)$ holds:
\begin{enumerate}
\item $\lam_1=\lam_2$ has an all-critical triangle $\Delta$,
it is not true that $C_1^1$ and $C_2^1$ are distinct edges of\, $\Delta$, and either $C_1^1\supset C_2^1$, or $C_2^1\supset C_1^1$.;
\item there is no all-critical triangle in $\lam_1\subset\lam_2$, and $C_1^j\supset C_2^j$ for $j=1,2$
(in particular, if $\lam_2$ is dendritic then $\lam_1=\lam_2$).
\end{enumerate}
\end{thm}

\begin{proof}
If the mixed tags of $(\lam_1,C_1^1,C_1^2)$ and $(\lam_2,C_2^1,C_2^1)$
are non-disjoint, then, by Lemma~\ref{l:interlink}, either
$\lam_1=\lam_2$ share an all-critical triangle $\Delta$, or these marked
laminations are linked or essentially equal. In the first case
consider several possibilities for the critical patterns.
One can immediately see that the only way the mixed tags are disjoint
is when  $C_1^1$ and $C_2^1$ are distinct edges of $\Delta$; since the
mixed tags are known to be non-disjoint we see that this corresponds to
case (1) from the theorem. In the second case the fact that our marked
laminations are linked or essentially equal implies, by
Theorem~\ref{t:noesli}, that case (2) of the theorem holds.
The opposite direction of the theorem follows from definitions.
\end{proof}

\subsection{Upper semi-continuous tags}

\begin{dfn}\label{d:usc}
A collection $\mathcal E=\{E_\alpha\}$ of compact and disjoint subsets of a metric space
$X$ is \emph{upper semi-continuous $($USC$)$} if, for every $E_\alpha$ and every open set
$U\supset E_\alpha$, there exists an open set $V$ containing $E_\alpha$ so that,
for each $E_\beta\in \mD$, if $E_\beta\cap V\ne\0$, then $E_\beta\subset U$.
A decomposition of a metric space is said to be \emph{upper semi-continuous $($USC$)$} if the corresponding collection of sets
is upper semi-continuous.
\end{dfn}

Upper semi-continuous decompositions are studied in \cite{dave86}.

\begin{thm}[\cite{dave86}]\label{t:dav}
If $\mathcal E$ is an upper se\-mi\-con\-ti\-nuous decomposition of a separable metric space $X$,
then the quotient space $X/\mathcal E$ is also a separable metric space.
\end{thm}

Before applying Theorem~\ref{t:dav} to our\, tags, we draw a
distinction between two notions. If $\sim$ is a dendritic
laminational equivalence relation, and $\lam_\sim$ has critical objects
$X, Y$ then by a \emph{marked laminational equivalence relation} we
mean a triple $(\sim, X, Y)$ or a triple $(\sim, Y, X)$ (if $X\ne Y$),
and just $(\sim, X, X)$ (if $X=Y$). Consider a marked geodesic
lamination $(\lam_\sim, C_1, C_2)$. Each marked laminational
equivalence relation is associated with the corresponding marked
geodesic lamination (the first of the two critical sets in a marked
laminational equivalence relation becomes $C_1$ and the second becomes
$C_2$). However if $\sim$ has an all-critical triangle $\Delta$ then
there are more possibilities for $(\lam_\sim, C_1, C_2)$ than
just $(\lam_\sim, \Delta, \Delta)$. E.g., $C_1, C_2$ could be two
distinct edges of $\Delta$. Still, mixed tags of laminational
equivalence relations are mixed tags of the corresponding geodesic
laminations and so our results obtained for geodesic laminations apply
to them.

Recall that the map $\ta_l$ was defined in Definition~\ref{d:siblita}.
To a marked laminational equivalence relation $(\sim,C,D)$, or to its
critical pattern $(C, D)$, the map $\ta_l$ associates the corresponding
\emph{mixed tag} $\ta_l(\sim, C, D)=\coc(C)\times \si_3(D)\subset
\overline{\disk}\times\overline{\disk}$.

\begin{thm}\label{t:tagusc}
The family $\{\ta_l(C^1,C^2)\}=\cmld$ of mixed tags of cubic marked dendritic
invariant laminational equivalence relations forms an upper
semi-continuous decomposition of the union $\cmld^+$ of all these tags.
\end{thm}

\begin{proof}
If $(\sim_1,C_1^1,C_1^2)$ and $(\sim_2,C_2^1,C_2^2)$ are marked
dendritic laminational equivalence relation, and $\ta_l(C_1^1,C_1^2)$ and
$\ta_l(C_2^1,C_2^2)$ are non-disjoint, then, by
Theorem~\ref{t:dendrilink} applied to the marked geodesic laminations
$(\lam_{\sim_1},$ $C_1^1,$ $C_1^2)$ and $(\lam_{\sim_1}, C_2^1,C_2^2)$, we
have that the corresponding marked laminational equivalence relations
are equal, i.e.
$(\lam_{\sim_1},C_1^1,C_1^2)=(\lam_{\sim_2},C_2^1,C_2^2)$. Hence the
family $\{\ta_l(C^1,$ $C^2)\}$ forms a decomposition of $\cmld^+$.

Suppose next that $(\sim_i,\zc_i)$ is a sequence of marked dendritic
laminational equivalence relations with $\zc_i=(C_i^1,C^2_i)$. Assume
that there is a limit point of the sequence of their tags
$\coc(C^1_i)\times \si_3(C^2_i)$ that belongs to the tag of a marked
dendritic laminational equivalence $(\sim_D, \zc_D)$ where
$\zc_D=(C^1_D, C^2_D)$. By
\cite{bmov13} and Lemma~\ref{l:decrilim}
we may assume that the sequence $(\lam_{\sim_i}, \zc_i)$ converges to a
marked lamination $(\lam_\infty, C_\infty^1, C_\infty^2)$ with critical
pattern $\fup_\infty=(C^1_\infty,C^2_\infty)$. By the assumption,
$\ta_l(\zc_D)\cap \ta(\fup_\infty)\ne\0$. By Theorem~\ref{t:dendrilink},
$\lam_D\subset \lam_\infty$ and $C_\infty^j\subset C_D^j$ for
$j=1,2$. Hence $\ta_l(\lam_\infty, \fup_\infty)\subset
\ta_l(\lam_D,\zc_D)$.
\end{proof}

Denote the quotient space of $\cmld^+$, obtained by collapsing every
element of $\cmld$ to a point, by $\md_3^{comb}$ (elements
of $\cmld$ are mixed tags of critical patterns of marked dendritic
laminational equivalence relations). Let $\pi:\cmld^+\to \md_3^{comb}$
be the quotient map. By Theorem~\ref{t:dav}, the topological space
$\md_3^{comb}$ is separable and metric.  We show that $\md_3^{comb}$
can be viewed as a combinatorial model for $\mathcal{MD}_3$. Recall
that the map $\Gamma:\md_3\to \mc(\cdisk)\times \mc(\cdisk)$ was defined
right before Lemma~\ref{l:uppers}.

\begin{thm}\label{t:polytags}
The composition $\pi\circ \ta_l\circ \Gamma:\md_3\to \md_3^{comb}$ is a continuous surjective map.
\end{thm}

\begin{proof}
By definition and Corollary~\ref{c:crista}, the map $\Gamma$ is upper
semi-continuous and surjective. Also, $\ta_l$ is continuous in the
Hausdorff topology and preserves inclusions. Finally, $\pi$ is
continuous by definition. Thus, $\pi\circ \ta_l\circ
\Gamma:\md_3\to \md_3^{comb}$ is a continuous surjective map as
desired.
\end{proof}



\end{document}